\documentclass[11pt]{amsart}

\usepackage{amssymb, amsmath, amsthm, mathrsfs}

 \newtheorem{thm}{Theorem}[section]
 \newtheorem{cor}[thm]{Corollary}
 
 \newtheorem{prop}[thm]{Proposition}
 \theoremstyle{definition}
 \newtheorem{defn}[thm]{Definition}
 \theoremstyle{remark}
 \newtheorem{rem}[thm]{Remark}
 
 \numberwithin{equation}{section}
\begin{document}
\pagestyle{plain}
{\raggedright{\footnotesize {The article was published in ``{Topics in Mathematics, Computer Science and Philosophy}, {\it A Festschrift for Wolfgang W. Breckner}'', Presa Universitar\u a Clujean\u a, 2008, pp. 107 -- 122. The present version contains some spelling and typing corrections and is differently formatted. For references please use the information above.}}}
\setcounter{page}{107}

\vskip2cm

\title[Dilations on locally Hilbert spaces]{Dilations on locally Hilbert spaces}

\author[D. Ga\c spar]{Dumitru Ga\c spar}



\author[P. Ga\c spar]{P\u astorel Ga\c spar}

\author[N. Lupa]{Nicolae Lupa}



\date{}

\dedicatory{Dedicated to Wolfgang W. Breckner for his 65th birthday}

\begin{abstract}
The principal theorem of Sz.-Nagy on dilation of a positive definite Hilbert space operator valued function has played a central role in the development of
the non-self-adjoint operator theory. In this paper we introduce the positive definiteness for locally Hilbert space operator valued kernels, we prove an analogue of the Sz.-Nagy dilation theorem
and, as application, we obtain dilation results for locally contractions and locally $\rho$~-~contractions as well as  for  locally semi-spectral measures.
\end{abstract}

\maketitle

\section{Introduction}

It is well known that the class  $\mathcal{B}(H)$  of all bounded linear operators on a Hilbert space $H$ can be organized as a $C^*$~-~algebra, and any $C^*$~-~algebra embeds isometrically in such an operator algebra. At the same time, because the above algebra $\mathcal{B}(H)$ is the dual of the trace-class $\mathscr{C}_{1}(H)$,  it follows that it is a $W^*$~-~algebra and conversely, any $W^*$~-~algebra can be identified up to an algebraic-topological  isomorphism with a weak operator closed $*$~-~subalgebra in $\mathcal{B}(H)$. Hence, the
spectral theory of bounded linear operators on a Hilbert space is developed in close connection with the theory of $C^*$~-~algebras. In recent times, a more general theory, namely that of locally $C^*$~-~algebras (\cite{In.1971}) and of locally $W^*$~-~algebras (\cite{Fr.1986}, \cite{Ma.1986}), is developed. Since such a locally convex $*$~-~algebra embeds in an algebra of continuous linear operators on a so-called locally Hilbert space and since the most important concepts (as selfadjointness, normality, positivity etc.) of the $C^*$~-~ and $W^*$~-~algebras theory extend to the frame of locally $C^*$~-~  and $W^*$~-~algebras, a self-adjoint spectral theory on locally Hilbert spaces can be easily developed (we refer to \cite{Ap.1971, Fr.1988, Sc.1975}).
This paper is intended to be an introduction to the non-selfadjoint spectral theory in this frame. More precisely, after completing some results on linear operators between locally Hilbert spaces (adjoint, isometries, partial isometries, contractions, unitary operators), we introduce reproducing kernel locally Hilbert spaces, we give a general dilation theorem for positive definite locally Hilbert space operator valued maps and, as consequences, we obtain dilation results for locally semi-spectral measures, locally ($\rho$~-)~contractions, semi-groups of locally contractions,  as well as extensions for isometries and subnormal operators in the setting of locally Hilbert spaces.

Let us now recall the basic definitions and results regarding a general locally $C^*$~-~algebra, a locally Hilbert space and the associated locally $C^*$~-~algebra of continuous operators on it.

If $A$ is a $*$~-~algebra (over $\mathbb{C}$), then a $C^*$~-~\textit{seminorm} $p$ on $A$ is a seminorm satisfying:
\begin{equation*}
p(a^* a) = p(a)^2 ,\ \ a \in A.
\end{equation*}
It was  proved, first by C. Apostol in \cite{Ap.1971}, that for a complete locally convex $*$~-~algebra $A$ and for a continuous $C^*$~-~seminorm $p$ on $A$, the quotient $*$~-~algebra
$A_p :\ = A/ N(p),\ \ N(p) = \ker p$ is a $C^*$~-~algebra. The set of all such $p$ ' s will be denoted by $S(A)$. Let us also note that such a $C^*$~-~seminorm $p$ satisfies (see \cite{Se.1979})
\begin{equation*}
p(ab) \le p(a) p(b) ,\ \ a, b \in A
\end{equation*}
(i.e. is submultiplicative), and
\begin{equation*}
p(a) = p(a^*) , \ a \in A
\end{equation*}
(i.e. is a $m^*$~-~seminorm).
Complete locally convex $*$~-~algebras endowed with the topology generated by a calibration consisting of all continuous $C^*$~-~seminorms were first studied by C. Apostol \cite{Ap.1971} and A. Inoue \cite{In.1971} in 1971. The latter, as well as M. Fragoulopoulou \cite{Fr.1988} later on, called these objects \emph{locally $C^*$~-~algebras}, whereas other authors \cite{Ap.1971, Sc.1975} called them \emph{$LMC^*$~-~algebras (locally multiplicative $*$~-~algebras)} or even \emph{pro~-~$C^*$~-~algebras} (see \cite{Am.2002, Ma.1986}). We shall adopt here the terminology \emph{$LC^*$~-~algebra} (see also \cite{Zh.Sh.2001}). We shall also suppose that such an $A$ has a unit. Note also that to each $LC^*$~-~algebra $A$ (endowed with the calibration $S(A)$) an inverse system of $C^*$~-~algebras can be attached (for example $\{ A_p , \pi_{p,q} \}_{p \le q}$, where $\pi_{p, q}$ is the natural embedding of $A_p$ into $A_q$), such that $A$ is the inverse (projective) limit of such a system ($A = \varprojlim\limits_{p \in S(A)} A_p$).
In fact inverse limit of any inverse system of $C^*$~-~algebras can stand for defining $LC^*$~-~algebras. Analogously, inverse limit of $W^*$~-~algebras were called \emph{locally $W^*$~-~algebras} or \emph{$LW^*$~-~algebras} (see \cite{Fr.1986, In.1971, Ma.1986, Sc.1975} a.o.). This is why many aspects of the selfadjoint spectral theory can be transposed from $C^*$~-~and $W^*$~-~algebras to $LC^*$~-~and $LW^*$~-~algebras, respectively (see \cite{Ap.1971} for the commutative case and \cite{Fr.1986, Fr.1988, In.1971, Jo.2006, Sc.1975} for the non-commutative one). So for an element $a$ from an $LC^*$~-~algebra $A$ we can define, as usually, its spectra $Sp(a)$ and the following assertions hold (see \cite{Am.2002, Ap.1971, Fr.1986, Fr.1988, In.1971, Ph.1988, Sc.1975}):
\begin{itemize}
\item[(i)] $a$ is (locally) self-adjoint (i.e. $a=a^*$), iff $Sp(a)\subset\mathbb{R}$; we shall denote that
$a\in A_h$;
\item[(ii)] $a$ is (locally) positive (i.e. $a=b^*b$, for some $b\in A$), iff $Sp(a)\subset [0,\infty)$; we denote that by $a\in A_+$; $A_+$ is a closed cone in $A$ and $A_+\cap - A_+=\{0\}$;
\item[(iii)] for a (locally) projection $a$ (i.e. $a=a^2\in A_h$), we denote $a\in\mathcal{P}(A)$; we  easily obtain that $\mathcal{P}(A)\subset A_+\subset A_h$;
\item[(iv)] $a$ is (locally) normal (i.e. $aa^*=a^*a$), we denote briefly  $a\in A_n$; it, evidently, holds $A_h\subset A_n$;
\item[(v)] if $a$ is a (locally) isometry (i.e. $a^*a=e$, where $e$ is the unit element of $A$), then  $aa^*\in\mathcal{P}(A)$ ($aa^*$ being the ``range'' projection);
\item[(vi)] if $a$ and $a^*$ are simultaneously (locally) isometries, then $a$ will be called a (locally) unitary element ($a\in\mathcal{U}(A)$); evidently $Sp(a)\subset\mathbb{T}$, where $\mathbb{T}$ is the torus and $\mathcal{U}(A)\subset A_n$;
\item[(vii)] $a\in A$ is a bounded element in $A$ (i.e. $\|a\|_b:=\sup\{p(a),p\in\mathcal{S}(A)\}<+\infty$, denoted  by $a\in A_b$), iff $Sp(a)$ is bounded in $\mathbb{C}$; evidently $\mathcal{U}(A)\subset A_b$; moreover $A_b$ endowed with the above mentioned sup-norm is a $\mathbb{C}^*$-algebra, which is dense in $A$;
\item[(viii)] any normal element has an integral representation with respect to a spectral measure (locally projection valued) on $\mathbb{C}$ (supported on the spectrum).
\end{itemize}

Now, having in view the above mentioned embedding of any $LC^*$~-~ or $LW^*$~-~algebra in a special $*$~-~algebra of continuous linear operators on a locally Hilbert space and that, in the Hilbert space frame, a non-self-adjoint spectral theory can be developed with the aid of dilation theory, we shall work with operators on locally Hilbert spaces, but also with operators between such spaces. Let's recall some precise definitions.

A \emph{locally Hilbert space} is a strict inductive limit of some ascending family (indexed after a directed set) of Hilbert spaces.
More precisely, given a directed (index) set $\Lambda$ and a family $\{H_\lambda\}_{\lambda\in\Lambda}$ of Hilbert spaces such that
\begin{equation}
\label{eq12} H_\lambda\subset H_\mu \text{ and } \langle\cdot,\cdot\rangle_\lambda=\langle\cdot,\cdot\rangle_\mu \text{ on } H_\lambda,
\end{equation}
(i.e. for $\lambda\leq\mu$, the natural embedding $J_{\lambda\mu}$ of $H_\lambda$ into $H_\mu$ is an isometry), then we endow the linear space $\mathbb{H}=\bigcup\limits_{\lambda\in\Lambda}H_\lambda$ with the inductive limit of $H_\lambda$ ($\lambda\in\Lambda$). Such an $\mathbb{H}$ will be called a \emph{locally Hilbert space} (associated to the ``inductive'' family $\{H_\lambda\}_{\lambda\in\Lambda}$). Recall that the \emph{inductive limit topology} on $\mathbb{H}$ is the finest one, for which the embeddings $J_\lambda$ of $H_\lambda$ into $\mathbb{H}$ are continuous, for all $\lambda\in\Lambda$.

Now we define the associated $LC^*$-algebra. Let $\{T_\lambda\}_{\lambda\in\Lambda}$ be an inductive system of bounded linear operators on $\{H_\lambda\}_{\lambda\in\Lambda}$  (i.e. $T_\lambda\in\mathcal{B}(H_\lambda)$ and
\begin{equation}\label{eq13}
T_\mu J_{\lambda\mu}=J_{\lambda\mu}T_\lambda,
\end{equation}
for each $\lambda,\mu\in\Lambda,$  $\lambda\leq\mu$), such that
\begin{equation}\label{eq14}
T_\mu P_{\lambda\mu}=P_{\lambda\mu}T_\mu,\ \lambda,\mu\in\Lambda,\ \lambda\leq\mu,
\end{equation}
where $P_{\lambda\mu}=J_{\lambda\mu}J_{\lambda\mu}^*$ is the
self-adjoint projection on $H_\mu$ with the range $H_\lambda$. Then
by putting $Th:=T_\lambda h, \, h\in H_\lambda,\,\lambda\in\Lambda$ we
have a correct definition of the linear operator $T$ on $\mathbb{H}$, which is also continuous (relative to the inductive limit topology on $\mathbb{H}$). We use the notation
$T=\varinjlim\limits_\lambda T_\lambda$, and $\mathcal{L}(\mathbb{H})$ for the algebra of all  operators $T$ as above.

Let us also note that a given linear operator $T$ on $\mathbb{H}=\varinjlim\limits_{\lambda\in\Lambda} H_\lambda$ is defined by an inductive system of linear operators  $T=\varinjlim\limits_\lambda T_\lambda$, iff it is invariant to each $H_{\lambda}$, $\lambda\in\Lambda$, i.e. it satisfies $range(TJ_\lambda)\subset H_\lambda$, $\lambda\in\Lambda$, the linear operator $T_\lambda$ from $H_\lambda$ into $H_\lambda$, being given by $T_{\lambda}h:=Th$, $h\in H_\lambda$, $\lambda\in\Lambda$. We also add that, in this case, $T$ is  continuous on $\mathbb{H}$ iff $T_\lambda\in\mathcal{B}(H_\lambda)$, $\lambda\in\Lambda$. Consequently, the following assertion holds:
\begin{itemize}
\item[(ix)] $ \mathcal{L}(\mathbb{H})$ consists exactly of those continuous linear operators on $\mathbb{H}$, which are invariant to each $H_\lambda$, $\lambda\in\Lambda$ and whose ``restrictions'' satisfy \eqref{eq14}.
\end{itemize}
Let us now remark that if $T=\varinjlim\limits_\lambda T_\lambda\in\mathcal{L}(\mathbb{H})$, then  $\{T^*_\lambda\}_{\lambda\in\Lambda}$ is an inductive system of
operators on $\bigcup\limits_\lambda H_\lambda$, satisfying (\ref{eq14}). Indeed,
$T^*_\lambda\in\mathcal{B}(H_\lambda)$,
($\lambda\in\Lambda$) and (\ref{eq14}) is equivalent (by passing to
the adjoint) to
\begin{equation}
\label{eq15} T^*_\mu P_{\lambda\mu}=P_{\lambda\mu}T^*_\mu ,\quad\lambda,\mu\in\Lambda,\, \lambda\leq\mu.
\end{equation}
Now (\ref{eq13}) and (\ref{eq14}) for the system $\{T_\lambda\}_{\lambda\in\Lambda}$ implies (\ref{eq13}) for $\{T^*_\lambda\}_{\lambda\in\Lambda}$ in the following manner.
For an arbitrary $h_\lambda\in H_\lambda$, by applying (\ref{eq15}), it holds that $T^*_\mu h_\lambda=T^*_\mu P_{\lambda\mu}h_\lambda=P_{\lambda\mu}T^*_\mu h_\lambda,\,\lambda\leq\mu$, and hence $T^*_\mu h_\lambda\in H_\lambda$.

By a straightforward computation, we have $(T^*_\mu
h_\lambda-T^*_\lambda h_\lambda,h_\lambda')_\lambda=(h_\lambda,T_\mu
h'_\lambda-T_\lambda h'_\lambda)_\lambda,\, h'_\lambda\in
H_\lambda,$ where, by (\ref{eq13}), the right hand site vanishes.
Now $T^*_\mu h_\lambda-T^*_\lambda h_\lambda=0, h_\lambda\in
H_\lambda$ which means (\ref{eq13}) for
$\{T^*_\lambda\}_{\lambda\in\Lambda}$. In this way we may define
$\varinjlim\limits_{\lambda\in\Lambda} T^*_\lambda:=T^*$  as an operator on $\mathbb{H}$. So $T\to T^*$ is an involution on $\mathcal{L}(\mathbb{H})$. It is now
a simple matter to check that $\mathcal{L}(\mathbb{H})$ is an $LC^*$~-~algebra
with the calibration $\{\|\cdot\|_\lambda\}_{\lambda\in\Lambda}$
defined by
\begin{equation}
\label{eq16} \|T\|_\lambda:=\|T_\lambda\|_{\mathcal{B}(H_\lambda)},\quad\lambda\in\Lambda.
\end{equation}

 \section{The space $\mathcal{L}(\mathbb{H}^1,\mathbb{H}^2)$}

Now, since for the dilation theorems the ``isometric'' embedding of a locally Hilbert space into another one is necessary, we extend the definition of $\mathcal{L}(\mathbb{H})$ to $\mathcal{L}(\mathbb{H}^1,\mathbb{H}^2)$ with two different locally Hilbert spaces $\mathbb{H}^1$ and $\mathbb{H}^2$, define the involution in this case and the locally partial isometries.

Given two families of Hilbert spaces $\{H^1_\lambda\}_{\lambda\in\Lambda}$ and $\{H^2_\lambda\}_{\lambda\in\Lambda}$,
indexed by the same directed set $\Lambda$ and satisfying (each of them) the condition (\ref{eq12}), we denote by $J_{\lambda\mu}^k$ the natural embeddings of $ H_\lambda^k$ into $H_\mu^k$, $\lambda\leq\mu$ and consider  $\mathbb{H}^k=\lim\limits_{\stackrel{\longrightarrow}{\lambda}}H_\lambda^k=\bigcup\limits_{\lambda\in\Lambda}H_\lambda^k,k=1,2$, the corresponding inductive limit.
Take also an inductive system of bounded linear operators $\{T_\lambda\}_{\lambda\in\Lambda}$ from $\bigcup\limits_{\lambda\in\Lambda}H^1_\lambda$ into $\bigcup\limits_{\lambda\in\Lambda}H^2_\lambda$ (i.e. $T_\lambda\in\mathcal{B}(H_\lambda^1,H^2_\lambda), \,\lambda\in\Lambda,$ and
\begin{equation}
 \label{eq21} T_\mu J_{\lambda\mu}^1=J_{\lambda\mu}^2 T_\lambda,
\end{equation}
for each $\lambda\leq\mu,\,\lambda,\mu\in\Lambda$), which also satisfies
\begin{equation}
 \label{eq22} T_\mu P_{\lambda\mu}^1=P_{\lambda\mu}^2 T_\mu,\,\,\lambda,\mu\in\Lambda,\,\lambda\leq\mu,
\end{equation}
where $P_{\lambda\mu}^k=J_{\lambda\mu}^k J_{\lambda\mu}^{k*}$ are self-adjoint projections in $H^k_\mu$, having the range $H^k_\lambda$ ($k=1,2$).

Now (\ref{eq21}) allows us to define correctly the operator $T$ through
\begin{equation}
 \label{eq23} Th=T_\lambda h,\quad h\in H_\lambda^1,\, \lambda\in\Lambda,
\end{equation}
as a linear operator from $\mathbb{H}^1$ into $\mathbb{H}^2$, which is continuous in the inductive limit topology. The class of this operators will be denoted by $\mathcal{L}(\mathbb{H}^1,\mathbb{H}^2)$. This is a complete locally convex space with the calibration consisting of the semi-norms defined as
\begin{equation}
 \label{eq24} \|T\|_\lambda:=\|T_\lambda\|_{\mathcal{B}(H^1_\lambda, H^2_\lambda)},\lambda\in\Lambda, T=\varinjlim\limits_{\lambda\in\Lambda} T_\lambda\in\mathcal{L}(\mathbb{H}^1, \mathbb{H}^2).
\end{equation}
It is clear that $\mathcal{L}(\mathbb{H}, \mathbb{H})=\mathcal{L}(\mathbb{H})$.

Now, returning to an operator $T$ from $\mathcal{L}(\mathbb{H}^1,\mathbb{H}^2)$, the relation (\ref{eq22}) for the inductive system $\{T_\lambda\}_{\lambda\in\Lambda}$ is equivalent (by passing to the hilbertian adjoint) to
\begin{equation}
 \label{eq25} T^*_\mu P_{\lambda\mu}^2=P_{\lambda\mu}^1T^*_\mu,\quad\lambda,\mu\in\Lambda,\,\lambda\leq\mu,
\end{equation}
which as in the first section and by (\ref{eq21}) implies
\begin{equation}
 \label{eq26} T^*_\mu J_{\lambda\mu}^2=J_{\lambda\mu}^1T^*_\lambda,\quad\lambda,\mu\in\Lambda,\,\lambda\leq\mu.
\end{equation}
Indeed, since $T_\mu^*\in\mathcal{B}(H^2_\mu,H^1_\mu)$, for an arbitrary $h_\lambda^2\in H_\lambda^2$ it holds
$$T_\mu^* h_\lambda^2=T_\mu^* P^2_{\lambda\mu} h_\lambda^2=P^1_{\lambda\mu} T_\mu^* h_\lambda^2,$$
hence $T_\mu^* h_\lambda^2\in H_\lambda^1$. Now $T_\mu^* h_\lambda^2-T_\lambda^* h_\lambda^2$ satisfies for arbitrary $h1_\lambda \in H^1_\Lambda$:
$$(T_\mu^* h_\lambda^2-T_\lambda^* h_\lambda^2,h_\lambda^1)=(h_\lambda^2,T_\mu h_\lambda^1-T_\lambda h^1_\lambda),$$
which, by (\ref{eq21}), vanishes. Since  $h_\lambda^1\in H_\lambda^1$ is arbitrary, it results $T_\mu^* h_\lambda^2-T_\lambda^* h_\lambda^2=0.$ Because $h_\lambda^2$ is arbitrary in $H_\lambda^2$, relation (\ref{eq26}) holds.\\
Defining
\begin{equation}
 \label{eq27}
 T^*:=\varinjlim\limits_{\lambda\in\Lambda} T_\lambda^*,
\end{equation}
we obtain $T^*\in\mathcal{L}(\mathbb{H}^2,\mathbb{H}^1)$ and, finally, the mapping
\begin{equation}
 \label{eq28}\mathcal{L}(\mathbb{H}^1,\mathbb{H}^2)\ni T\longmapsto T^*\in \mathcal{L}(\mathbb{H}^2,\mathbb{H}^1)
\end{equation}
satisfies the properties of the adjunction (as in the case of Hilbert space operators from one space to another).
 For the adjunction of a product let us observe that if we have three locally Hilbert spaces
 $\mathbb{H}^k=\varinjlim\limits_{\lambda}H_\lambda^k=\bigcup\limits_{\lambda\in\Lambda}H_\lambda^k$
 ($k=1,2,3$), and $T=\varinjlim\limits_\lambda T_\lambda\in\mathcal{L}(\mathbb{H}^1,\mathbb{H}^2)$,
 $S=\varinjlim\limits_\lambda S_\lambda\in\mathcal{L}(\mathbb{H}^2,\mathbb{H}^3)$,  by (\ref{eq21}), for
 $T$ and $S$ we successively have $$S_\mu T_\mu J^1_{\lambda\mu}=S_\mu J^2_{\lambda\mu} T_\lambda=J^3_{\lambda\mu}S_\mu T_\mu,\,\lambda,\mu\in\Lambda,\,\lambda\leq\mu,$$
from where $\{S_\lambda T_\lambda\}_{\lambda\in\Lambda}$ satisfy (\ref{eq21}) as operators from
 $\mathcal{B}(H_\lambda^1,H_\lambda^3)$. Analogously from (\ref{eq22}) for $T$ and $S$ we infer (\ref{eq22}) for $ST$.
 Consequently $ST$ defined by $\varinjlim\limits_\lambda S_\lambda T_\lambda$ belongs to
 $\mathcal{L}(\mathbb{H}^1,\mathbb{H}^3)$, and is in fact the composition operator of $T$ and $S$.

Let us also observe that the corresponding semi-norms from $\mathcal{L}(\mathbb{H}^1,\mathbb{H}^3)$ satisfy
\begin{equation}
 \label{eq29} \|ST\|_\lambda\leq\|S\|_\lambda\|T\|_\lambda,T\in\mathcal{L}(\mathbb{H}^1,\mathbb{H}^2),S\in\mathcal{L}(\mathbb{H}^2,\mathbb{H}^3),\lambda\in\Lambda.
\end{equation}

In a similar way it is possible to define the composition $T^*S^*$ as a member of $\mathcal{L}(\mathbb{H}^3,\mathbb{H}^1)$. Regarding  both constructions, applying the adjunction of a product of Hilbert space operators, we get
\begin{equation}
  \label{eq210} T^*S^*=\varinjlim\limits_{\lambda\in\Lambda} T_\lambda^* S^*_\lambda =  \varinjlim\limits_{\lambda\in\Lambda} (S_\lambda  T_\lambda )^*= (ST)^*.
\end{equation}

Noticing that for $T\in\mathcal{L}(\mathbb{H}^1,\mathbb{H}^2)$ it is possible to form $T^*T\in\mathcal{L}(\mathbb{H}^1)$ and $TT^*\in\mathcal{L}(\mathbb{H}^2)$. These are clearly self-adjoint elements in the corresponding $LC^*$-algebras, the semi-norms $\|\cdot\|_\lambda$ satisfying the property $$\|T^*T\|_\lambda=\|T\|^2_\lambda=\|T^*\|^2_\lambda,\,\lambda\in\Lambda$$ (where the semi-norms are in $\mathcal{L}(\mathbb{H}^1)$, $\mathcal{L}(\mathbb{H}^1,\mathbb{H}^2)$ and $\mathcal{L}(\mathbb{H}^2,\mathbb{H}^1)$, respectively).

Having in view the above construction, the following characterizations of special elements in $\mathcal{L}(\mathbb{H})$ are immediate and the proof will be omitted.
\begin{prop}\label{prop21}
Let $\mathbb{H}=\varinjlim\limits_{\lambda\in\Lambda}H_\lambda$ be a locally Hilbert space and $T=\varinjlim\limits_{\lambda}T_\lambda$ be an element in $\mathcal{L}(\mathbb{H})$.
Then
  \begin{itemize}
  \item[(i)] $T$ is  locally self-adjoint  on $\mathbb{H}$, iff each $T_\lambda$ is self-adjoint on $H_\lambda$  $(\lambda\in \Lambda)$;
  \item[(ii)] $T$ is  locally positive  on $\mathbb{H}$, iff each $T_\lambda$ is positive on $H_\lambda$ $(\lambda\in \Lambda)$;
  \item[(iii)] $T$ is a locally projection  on $\mathbb{H}$, iff each $T_\lambda$ is a projection on $H_\lambda$ $(\lambda\in \Lambda)$;
  \item[(iv)] $T$ is locally normal  on $\mathbb{H}$, iff each $T_\lambda$ is normal on $H_\lambda$  $(\lambda\in \Lambda)$;
  \item[(v)] $T$ is a local isometry on $\mathbb{H}$ (i.e. $T^*T=I_\mathbb{H}$), iff each $T_\lambda$ is an isometry on $H_\lambda$ $(\lambda\in \Lambda)$.
  \end{itemize}
\end{prop}
Now, it is possible to define a locally partial isometry between two locally Hilbert spaces. Namely $V$ is a \emph{locally partial isometry}, when it is an operator acting between two locally Hilbert spaces $\mathbb{H}^1$ and $\mathbb{H}^2$, $V\in\mathcal{L}(\mathbb{H}^1,\mathbb{H}^2)$ and $V^*V$ is a locally projection on $\mathbb{H}^1$ (i.e. $V^*V\in\mathcal{P}(\mathcal{L}(\mathbb{H}^1))$).

Let us note that, as in the Hilbert space case, if $V\in\mathcal{L}(\mathbb{H}^1,\mathbb{H}^2)$ is a locally partial isometry, then $VV^*$ is a locally projection (on $\mathbb{H}^2$) as well. If $V^*V=1_{\mathbb{H}^1}$, then $V$ will be called a \emph{locally isometry} (from $\mathbb{H}^1$ to $\mathbb{H}^2$). In the case in which $VV^*=1_{\mathbb{H}^2}$,  $V$ will be called a \emph{locally co-isometry}. A locally isometry, which is also a locally co-isometry is a \emph{locally unitary operator} from $\mathbb{H}^1$ into $\mathbb{H}^2$.

The following characterizations are also easy to prove:

\begin{thm}\label{thm22}
Let $V=\varinjlim\limits_\lambda V_\lambda$ be an element from $\mathcal{L}(\mathbb{H}^1,\mathbb{H}^2)$. Then
  \begin{itemize}
  \item[(i)] $V$ is a locally partial isometry, iff each $V_\lambda$ is a partial isometry from $H^1_\lambda$ into $H^2_\lambda$ ($\lambda\in\Lambda$);
  \item[(ii)] $V$ is a locally co-isometry, iff each $V_\lambda$ is a co-isometry from $H^1_\lambda$ into $H^2_\lambda$ ($\lambda\in\Lambda$);
  \item[(iii)] $V$ is an invertible operator, iff each $V_\lambda$ is invertible ($\lambda\in\Lambda$). In this case $V^{-1}=\varinjlim\limits_{\lambda\in\Lambda}V_\lambda^{-1}\in\mathcal{L}(\mathbb{H}^2,\mathbb{H}^1)$;
  \item[(iv)] $V$ is a locally unitary operator from $\mathbb{H}^1$ onto $\mathbb{H}^2$, iff each $V_\lambda$ is a unitary operator from $H_\lambda^1$ onto $H^2_\lambda$ ($\lambda\in\Lambda$);
  \item[(v)] \emph{(Fuglede-Putnam Theorem)} Let $N_1\in\mathcal{L}_n(\mathbb{H}^1)$ (locally normal operator) and $N_2\in\mathcal{L}_n(\mathbb{H}^2)$. If there exists  $S\in\mathcal{L}(\mathbb{H}^1,\mathbb{H}^2)$ such that $SN_1=N_2S$, then $SN_1^*=N^*_2S$.
  \end{itemize}
\end{thm}

Now it is  interesting to observe that the notion of orthogonally  closed subspace has a correspondent in the frame of locally Hilbert spaces. Indeed, we can give the following definition:
\begin{defn}\label{def21}
 A subspace $\mathbb{H}^1$ of a locally Hilbert space $\mathbb{H}$ is \emph{orthogonally complementable}, if there is a locally self-adjoint projection $P \in \mathcal{L}(\mathbb{H})$, such that $P\mathbb{H} = \mathbb{H}^1$.
\end{defn}
It is clear that any such ``orthogonally'' complementable  subspace is closed. For now we are not interested in the problem whether each closed subspace $\mathbb{H}^1$ is orthogonally complementable. However it is interesting to see that each $H_{\lambda_0} \ (\lambda_0 \in \Lambda)$ is orthogonally  complementable in $\mathbb{H} = \varinjlim\limits_{\lambda \in \Lambda} H_\lambda$. This is easily seen if we regard
$H_{\lambda_0}$ as the strict inductive limit of $H_{\lambda_0} \cap H_{\lambda}, \ \lambda \in \Lambda$, i.e. $H_{\lambda_0} = \varinjlim\limits_{\lambda \in \Lambda} H_{\lambda_0} \cap H_\lambda$.
Is is easy to obtain that the family $\{ H_{\lambda_0} \cap H_\lambda , \ \lambda\in\Lambda \}$ satisfies the condition \eqref{eq12} and $H_{\lambda_0} = \bigcup\limits_{\lambda\in\Lambda} (H_{\lambda_0} \cap H_\lambda)$. Moreover,  the natural embedding $J_{\lambda_0}$ of $H_{\lambda_0}$ into $\mathbb{H}$ satisfies the conditions \eqref{eq21} and \eqref{eq22}, if we consider $J_{\lambda_0} =
\varinjlim\limits_{\lambda \in \Lambda} J_{\lambda_0}^\lambda$, where $J_{\lambda_0}^\lambda$ is the natural embedding of $H_{\lambda_0} \cap H_\lambda$ into $H_\lambda$, $\lambda \in \Lambda$.
So $J_{\lambda_0}$ is a locally isometric operator from $\mathcal{L}(H_{\lambda_0}, \mathbb{H})$. In this way $J_{\lambda_0}J^*_{\lambda_0} \in \mathcal{L} (\mathbb{H})$ and $J_{\lambda_0}J^*_{\lambda_0} \mathbb{H} = H_{\lambda_0}$,
$J_{\lambda_0}J^*_{\lambda_0}$ being the desired locally  self-adjoint projection. So we have proved:
\begin{prop}\label{prop23}
If $\ \mathbb{H}= \varinjlim\limits_{\lambda \in \Lambda} H_\lambda$ is a locally Hilbert space, then $J_\lambda$ is a locally isometry from $H_\lambda$ into $\mathbb{H}$ and each $H_\lambda$, $\lambda\in\Lambda$ is an orthogonally complementable subspace in $\mathbb{H}$. Moreover, if $\,T\in\mathcal{L}(\mathbb{H})$, then $T= \varinjlim\limits_{\lambda \in \Lambda} T_\lambda$, where $T_\lambda=J_\lambda^*TJ_\lambda$ and $TJ_\lambda=J_\lambda J_\lambda^*TJ_\lambda=J_\lambda T_\lambda,\ \lambda\in\Lambda$.
\end{prop}

 \section{Locally positive definite operator valued kernels}

Let us mention that the first two named authors have introduced in \cite{Ga.Ga.2007} the positive definiteness for $LC^*$~-~algebra valued kernels. Recalling this definition for the $LC^*$~-~algebra $\mathcal{L}(\mathbb{H})$, we shall give a characterization of this positive definiteness in terms of elements of $\mathbb{H}$.
We start with a remark regarding the existence of a natural scalar product on a locally Hilbert space $\mathbb{H}=\varinjlim\limits_{\lambda\in\Lambda}H_\lambda$.
For a pair $h,k\in \mathbb{H}$, we put
\begin{equation}
 \label{eq31}\langle h,k\rangle:=\langle h,k\rangle_\lambda,
\end{equation}
where $\lambda\in\Lambda$ is chosen such that $h,k\in H_\lambda$. From the condition \eqref{eq12} it is easy to see that the definition \eqref{eq31} is correct (does not depend on the choice of $\lambda$) and satisfies the properties of a scalar product.

\begin{defn}[\cite{Ga.Ga.2007}]\label{def31}
Let $\mathbb{H} = \varinjlim\limits_{\lambda\in\Lambda} H_\lambda$ be a locally Hilbert space and $\mathcal{L}(\mathbb{H})$ be the previously defined $LC^*$~-~algebra. An $\mathcal{L}(\mathbb{H})$~-~valued kernel on a set $S$ (i.e. a function $\Gamma:S\times S \to \mathcal{L}(\mathbb{H})$) is a \emph{locally  positive definite kernel} (LPDK), if for each finitely supported function
\begin{equation}\label{eq32}
S\ni s \mapsto T_s = \varinjlim\limits_{\lambda\in\Lambda} T_s^\lambda \in \mathcal{L} (\mathbb{H})
\end{equation}
it holds
\begin{equation} \label{eq33}
\sum\limits_{s, t} T^*_t \Gamma(s, t) T_s \geq 0.
\end{equation}
\end{defn}
Looking at the condition \eqref{eq33} and using the scalar product \eqref{eq31} we shall deduce that it is equivalent to
\begin{equation}\label{eq34}
\sum\limits_{s,t}\langle\Gamma(s,t)h_s,h_t\rangle \geq 0,
\end{equation}
for any finitely supported function $S \ni s \mapsto h_s  \in \mathbb{H}$.  

Indeed, by Proposition \ref{prop21} (iii) and Proposition \ref{prop23}, \eqref{eq33} is equivalent to
\begin{equation}\label{eq3-1}
\sum_{s,t} T^{\lambda *}_t \Gamma^\lambda(s, t) T^\lambda_s\geq 0, \, \lambda\in\Lambda,
\end{equation}
which by the last part of Proposition \ref{prop23} is equivalent to
\begin{equation} \label{eq35}
\sum\limits_{s,t} T_t^{\lambda *} \Gamma^\lambda (s,t) T_s^\lambda \ge 0,\ \lambda \in \Lambda,
\end{equation}
which in turn, by the characterization of operatorial positive definiteness in the Hilbert space $H_\lambda$ is equivalent to
\begin{equation}\label{eq36}
\sum\limits_{s,t} \langle \Gamma^\lambda (s,t),h_s^\lambda, h_t^\lambda \rangle_\lambda \ge 0,\ \lambda \in \Lambda,
\end{equation}
which is obviously equivalent to \eqref{eq34} for each finitely supported ${\mathbb H}$~-~valued function $s\mapsto h_s$. \\
We have thus proven

\begin{thm} \label{thm31}
An $\mathcal{L}(\mathbb{H})$~-~valued kernel $\,\Gamma$ is an LPDK on $S$ iff for each finitely supported $\mathbb{H}$~-~valued function  $s\mapsto h_s$ on $S$, relation \eqref{eq34} is fulfilled.
\end{thm}

\begin{defn} \label{def32}
Let $S$ be a (commutative) $*$~-~semigroup with a neutral element $e$. An $\mathcal{L}(\mathbb{H})$~-~valued mapping $\varphi$ on $S$ is a \emph{locally positive definite function} (LPDF) on $S$ if the associated kernel
$\Gamma_\varphi$ defined by $\Gamma_\varphi (s, t):\ = \varphi(t^* s),$ $s, t \in S$ is an LPDK.
\end{defn}
From Theorem \ref{thm31} we immediately infer
\begin{cor} \label{cor32}
An $\mathcal{L} (\mathbb{H})$~-~valued function $\varphi$ on a $*$~-~semigroup $S$, where $\mathbb{H} = \varinjlim\limits_{\lambda} H_\lambda$ is a locally Hilbert space, is an LPDF on $S$, iff for each finitely supported function $ S\ni s \mapsto h_s \in \mathbb{H}$ it holds
\begin{equation} \label{eq37}
\sum_{s, t}\langle{\varphi(t^*s)h_s},{h_t} \rangle\geq 0.
\end{equation}
\end{cor}
If we also look at the ``localization'' of all operators which occur in the above considerations, we easily deduce
\begin{cor} \label{cor33}
Let $\mathbb{H} = \varinjlim\limits_\lambda H_\lambda$ be a locally Hilbert space and $\Gamma: S \times S \to \mathcal{L} (\mathbb{H}),$ $\Gamma(s, t) = \varinjlim\limits_{\lambda \in \Lambda} \Gamma^\lambda (s, t), \ s, t \in S$ be a kernel on the set $S$. The following conditions are equivalent:
\begin{itemize}
\item[(i)]{ $\Gamma$ is an $\mathcal{L}(\mathbb{H})$~-~valued LPDK on $S$;}
\item[(ii)]{ for each $\lambda \in\Lambda$, $\Gamma^\lambda$ is a  $\mathcal{B}(H_\lambda)$~-~valued PDK on $S$;}
\item[(iii)]{for each finitely supported $\mathbb{H}$~-~valued function $s\mapsto h_s$ on $S$,  relation \eqref{eq34} holds.}
\end{itemize}
\end{cor}
\begin{cor}
For an $\mathcal{L}(\mathbb{H})$~-~valued function on the $*$~-~semigroup $S$, the following conditions are equivalent:
\begin{itemize}
\item[(i)]{$\varphi$ is an $\mathcal{L}(\mathbb{H})$~-~valued LPDF on $S$;}
\item[(ii)]{for each $\lambda\in\Lambda$, $\varphi^\lambda$ is a $\mathcal{B}(H_\lambda)$~-~valued PDF on $S$;}
\item[(iii)]{$\varphi$ satisfies the condition \eqref{eq37}.}
\end{itemize}
\end{cor}

\section{Reproducing kernel locally Hilbert spaces}

In \cite{Ga.Ga.2007} we have defined the reproducing kernel Hilbert module over an $LC^*$~-~algebra $C$. This works for the $LC^*$~-~algebra $\mathcal{L}(\mathbb{H})$ as well. But for  $\mathcal{L}(\mathbb{H})$~-~valued kernels, analogue to the case of Hilbert spaces, it is also possible to introduce the reproducing kernel locally Hilbert space, whose reproducing kernel is $\mathcal{L}(\mathbb{H})$~-~valued.
\begin{defn} \label{def41}
Let $\mathbb{H} = \varinjlim\limits_\lambda H_\lambda$ be a fixed locally Hilbert space, $S$ be an arbitrary set and $\mathbb{K} = \varinjlim\limits_{\lambda \in \Lambda} K_\lambda $ be a locally Hilbert space consisting of $\mathbb{H}$~-~valued functions on $S$. $\mathbb{K}$ is called a \emph{reproducing kernel locally Hilbert space} (RKLHS), if there exists an $\mathcal{L}(\mathbb{H})$~-~valued kernel $\Gamma$ on $S$ such that the operators $\Gamma_s$, $s \in S$, between ${\mathbb H}$ and ${\mathbb K}$, defined by $$(\Gamma_s h)(t) :\ = \Gamma(s, t) h ;\ t \in S ,\ h \in \mathbb{H}$$ satisfy the following conditions:
\begin{itemize}
\item[(IP)]{$\Gamma_s \in \mathcal{L} (\mathbb{H}, \mathbb{K}) ,\ s \in S$ ;}
\item[(RP)]{$k(s) = \Gamma_s^* k ,\ k \in \mathbb{K} , \ s \in S$.}
\end{itemize}
\end{defn}
\begin{rem}
If an $\mathcal{L}(\mathbb{H})$~-~valued kernel on $S$, with the above property exists, then it is uniquely determined by $\mathbb{K}$. Indeed, if another $\mathcal{L}(\mathbb{H})$~-~valued kernel $\Gamma'$ satisfying the properties (IP) and (RP) exists, then (RP) implies $\Gamma_s^* k = {\Gamma_s'}^*k ,\ k \in \mathbb{K}$ wherefrom $\Gamma^*_s = {\Gamma'}^*_s$ as operators from $\mathcal{L}(\mathbb{H} , \mathbb{K})$. This implies now that $\Gamma(s, t) = \Gamma'(s,t);\ s,t \in S$. This is why, $\Gamma$ will be also called \emph{the locally reproducing kernel (LRK) of the RKLHS $\mathbb{K}$}. It will be also denoted by $\Gamma_\mathbb{K}$.
\end{rem}
\begin{rem}
If $\mathbb{H} = \varinjlim\limits_{\lambda\in\Lambda} H_\lambda$ is a locally Hilbert space and $\Gamma$ is an $\mathcal{L} (\mathbb{H})$~-~valued LRK for the locally Hilbert space $\mathbb{K} = \varinjlim\limits_{\lambda\in\Lambda} K_\lambda$, then, having in view that $\Gamma(s, t) = \varinjlim\limits_{\lambda\in\Lambda} \Gamma^\lambda (s,t) ,\ s, t \in S$ (as elements of $\mathcal{L}(\mathbb{H})$ !), the following properties are fulfilled:
\begin{itemize}
\item[(LIP)]{$\Gamma^\lambda_s h \in K_\lambda ,\,h\in H_\lambda,\ s\in S , \ \lambda \in \Lambda$;}
\item[(LRP)]{$\langle{k(s)},{h}\rangle_{H_\lambda} = \langle{k},{\Gamma_s^\lambda h}\rangle_{K_\lambda},\ h \in H_\lambda,\ k \in K_\lambda,\ s\in S , \ \lambda \in \Lambda$.}
\end{itemize}
This results by applying the definition of $\mathcal{L}(\mathbb{H}, \mathbb{K})$.
\end{rem}
It is now easily seen that the locally conditions (LIP) and (LRP) are sufficient to define  $\mathbb{K} = \varinjlim\limits_{\lambda\in\Lambda} K_\lambda$ as RKLHS with $\Gamma=\varinjlim\limits_{\lambda\in\Lambda} \Gamma^\lambda$ as LRK on $S$.

Moreover, we obtain
\begin{cor}\label{cor43}
The locally Hilbert space $\mathbb{K} = \varinjlim\limits_{\lambda\in \Lambda} K_\lambda$ of $\mathbb{H}$~-~valued functions on $S$ is a RKLHS, iff for each $\lambda \in \Lambda$, the Hilbert space $K_\lambda$ of $H_\lambda$~-~valued functions on $S$ is a reproducing kernel Hilbert space (RKHS). Moreover, if $\Gamma_\mathbb{K}$ is the RK of $\mathbb{K}$ and $\Gamma_{K_\lambda}$ is the RK of $K_\lambda$, then, for each
$s, t \in S$, we have
$
\Gamma_\mathbb{K} (s, t) = \varinjlim\limits_{\lambda\in\Lambda} \Gamma_{K_\lambda} (s, t).
$
In other words, $\Gamma_{K_\lambda} = \Gamma_{\mathbb{K}}^\lambda \ (\lambda\in\Lambda)$.
\end{cor}
\begin{prop}
If $\,\Gamma = \Gamma_\mathbb{K}$ is a LRK for a locally Hilbert space $\mathbb{K}$ of $\,\mathbb{H}$~-~valued functions on $S$, then $\Gamma$ is an $\mathcal{L}(\mathbb{H})$~-~valued LPDK on $S$.
\end{prop}
 The proof runs on the components $\Gamma^\lambda \ (\lambda\in\Lambda)$ of $\Gamma$, as in the corresponding Hilbert space case. \\
A more important result is:
\begin{thm}
Let $\mathbb{H} = \varinjlim\limits_{\lambda\in\Lambda} H_\lambda$ be a locally Hilbert space and $S$ be an arbitrary set. Then $\Gamma$ is an LRK for some locally Hilbert space $\mathbb{K} = \varinjlim\limits_{\lambda\in \Lambda} K_\lambda$ of $\mathbb{H}$~-~valued functions on $S$, iff it is an $\mathcal{L}(\mathbb{H})$~-~valued LPDK on $S$.
\end{thm}
\begin{proof}
It  remains to prove that, for a given $\Gamma$ as above, there exists a locally Hilbert space $\mathbb{K} = \varinjlim\limits_{\lambda\in\Lambda} K_\lambda$, consisting of $\mathbb{H}$~-~valued functions on $S$, which is an RKLHS, for which $\Gamma =\Gamma_\mathbb{K}$. First, since $\Gamma (s, t) = \varinjlim\limits_{\lambda\in\Lambda} \Gamma^\lambda (s,t),\ s, t \in S$, from the condition of LPD, it results that, for each $\lambda \in \Lambda$, $\Gamma^\lambda$ is a $\mathcal{B}(H_\lambda)$~-~valued PDK on $S$ (see \cite{Ga.1970}). Denote $K_\lambda$ the RKHS, with $\Gamma^\lambda$ as RK. From the properties of $\mathcal{L}(\mathbb{H})$ it results that the family $\{ K_\lambda ,\ \lambda \in \Lambda \}$ satisfies the condition for the construction of the inductive limit $\varinjlim\limits_{\lambda \in \Lambda} K_\lambda = \bigcup\limits_{\lambda \in \Lambda} K_\lambda$. Then $\mathbb{K} = \bigcup\limits_{\lambda \in \Lambda} K_\lambda$ is the desired locally Hilbert space.
\end{proof}

\section{Dilations of LPD operator valued functions on $*$~-~semigroups}

We are now in a position to prove, in the frame of operators on locally Hilbert spaces, the analogue of the famous dilation theorem of B. Sz.-Nagy (\cite{Sz.1955}). \\
Let $S$ be an abelian $*$~-~semigroup with the neutral element $e$. A \emph{representation} of $S$ on a locally Hilbert space $\mathbb{K}$ is an algebra morphism $\pi$ from $S$ into $\mathcal{L}(\mathbb{K})$, i.e.
\begin{align*}
\pi(e) &= I_\mathbb{K} \\
\pi (st) & = \pi(s) \pi(t)\\
\pi(s^*) & = \pi(s)^* ,\,s,t\in S.
\end{align*}
It is clear that such a representation generates through
$$
\Gamma_\pi (s, t): \ = \pi (t^*s) ,\ \ s, t \in S
$$
an $\mathcal{L}(\mathbb{K})$~-~valued LPDK on $S$. The converse doesn't hold in general. However an $\mathcal{L}(\mathbb{H})$~-~valued LPDF on $S$ is extensible in some sense to a larger locally Hilbert space. More precisely it holds:
\begin{thm} \label{thm51}
Let $S$ be a $*$~-~semigroup, $\mathbb{H}$ be a locally Hilbert space and $s\mapsto \varphi(s)$ be an $\mathcal{L}(\mathbb{H})$~-~valued function on $S$, which is a LPDF and satisfies the following boundedness condition:
\begin{itemize}
\item[(LBC)]{for each $u \in S$ and  $\lambda \in \Lambda$ there exists a positive constant $C_u^\lambda > 0$, such that
    \begin{equation*}
    \sum_{s, t} \langle{\varphi(t^*u^*us)h_s},{h_t}\rangle_{\lambda} \leq (C_u^\lambda)^2 \sum \langle{\varphi (t^*s) h_s},{h_t}\rangle_{\lambda},
    \end{equation*}}
 where   $\{ h_s \}_{s \in S}$ is an arbitrary finitely supported $H_\lambda$~-~valued function.
\end{itemize}
Then there exists a locally Hilbert space $\mathbb{K}$, in which $\mathbb{H}$ is naturally embedded by a locally isometry $J\in\mathcal{L}(\mathbb{H},\mathbb{K})$ and a representation $\pi$ of $S$ on $\mathbb{K}$, such that
$$ \varphi(s) = J^* \pi(s) J ,\ s \in S.$$
Moreover, it is possible to choose $\mathbb{K}$ satisfying the minimality condition
$$\bigvee\limits_{s \in S} \pi(s) \mathbb{H} = \mathbb{K} $$
and in this case $\mathbb{K}$ is uniquely determined up to a locally unitary operator. The following conditions will hold as well:
\begin{itemize}
\item[(i)]{$||\pi(u)||_\lambda \le C_u^\lambda , \ \lambda \in\Lambda, u\in S$;}
\item[(ii)]{$\varphi(sut)+\varphi(svt) = \varphi(swt)$, for each $s, t\in S$ implies $\pi(w) = \pi(u) + \pi (v)$.}
\end{itemize}
\end{thm}
\begin{proof}
By defining $\Gamma_\varphi (s,t) :\ = \varphi (t^* s)$, we infer that $\Gamma_\varphi$ is an LPDK on $S$. Then the desired locally Hilbert space will be $\mathbb{K} = \mathbb{K}_{\Gamma_\varphi}$, the RKLHS with $\Gamma_\varphi$ as LRK. As it is known, a dense subspace in $\mathbb{K}$ is given by
\begin{equation*}
\biggl\{ \sum\limits_s \Gamma^\varphi_s h_s , \text{ where } s \mapsto h_s \text{ is a finitely supported } \mathbb{H}\text{-valued function on } S\biggr\}.
\end{equation*}
The operators $J$ will be defined as
\begin{equation*}
J h = \sum_s \Gamma^\varphi_s h_s ,\ \text{ where } h_e = h, \ \text{and} \ h_s = 0, \ \text{for} \ s \ne e,
\end{equation*}
whereas the representation $\pi$, will be
\begin{equation*}
\pi (u) \sum_s \Gamma^\varphi_s h_s = \sum_s \Gamma^\varphi_{us} h_s .
\end{equation*}
With the prerequisites of the preceding sections it is now easy to verify the statements  (i) and (ii).
\end{proof}
The representation $\pi$ is called a \emph{minimal dilation} of the function $\varphi$. It is known in the frame of Hilbert space operators the notion of a minimal $\rho$~-~dilation (\cite{Ga.1970}).
A representation $\pi$ of $S$ on a locally Hilbert space $\mathbb{K}$, containing another $\mathbb{H}$ as a subspace, is called a $\rho$~-~dilation ($\rho > 0$) for an $\mathcal{L}(\mathbb{H})$~-~valued function $\varphi$ on $S$ if
$$ \varphi(s) = \rho J^* \pi(s) J ,\ s\in S\setminus\{e\}, $$
where $J$ is the natural (locally isometric) embedding of $\mathbb{H}$ into $\mathbb{K}$.
It is not hard to characterize the $\mathcal{L}(\mathbb{H})$~-~valued functions $\psi$ on a $*$~-~semigroup $S$, which admit $\rho$~-~dilations. Indeed, it holds:
\begin{thm} \label{thm52}
An $\mathcal{L}(\mathbb{H})$~-~valued function $\psi$ on $S$ has a $\rho$~-~dilation, iff the following conditions are fulfilled
\begin{itemize}
\item[($\rho$LPD)]{ $\rho \sum\limits_{s,t: t^*s=e} \langle{h_s},{h_t}\rangle + \sum\limits_{s, t : t^*s \ne e} \langle{\psi(t^*s)h_s},{h_t}\rangle \ge 0$, for each finitely supported $\mathbb{H}$~-~valued function $s \mapsto h_s$ on $S$;}
\item[($\rho$LBC)]{for each $\lambda\in\Lambda$ and each $u\in S$, there is a constant $C_u^\lambda > 0$, such that
\begin{multline*}
\rho \sum\limits_{s, t : t^*u^*us = e} \langle{h_s},{h_t}\rangle_{\lambda} +
\sum\limits_{s, t : t^*u^*us \ne e} \langle{\psi (t^*u^*us)h_s},{h_t}\rangle_{\lambda} \leq \\
(C_u^\lambda)^2 \left[ \rho \sum_{s,t : t^*s = e} \langle{h_s},{h_t}\rangle_{\lambda} + \sum_{s, t : t^*s \ne e}\langle{\psi(t^*s)h_s},{h_t}\rangle_{\lambda} \right]
\end{multline*}
for each finitely supported $H_\lambda$~-~valued function $s \mapsto h_s$ on $S$.}
\end{itemize}
It is also possible to have the minimality condition and analogue properties for the $\rho$~-~dilation as in the previous theorem.
\end{thm}
\begin{proof}
By putting
\begin{equation*}
\varphi (s) =
\begin{cases}
\frac{1}{\rho} \psi (s) , & s \ne e \\
I_\mathbb{H} , & s = e
\end{cases},
\end{equation*}
we have that $\varphi$ satisfies the conditions of Theorem \ref{thm51} and the dilation of the function $\varphi$ will be a $\rho$~-~dilation for the given function $\psi$.
\end{proof}

\section{Applications}

\paragraph1 It is now possible to dilate the \emph{locally positive and ``bounded'' $\mathcal{L} (\mathbb{H})$~-~valued measure} on a $\sigma$~-~algebra $\Sigma$, to a multiplicative locally projection $\mathcal{L}(\mathbb{K})$~-~valued multiplicative measure, where $\mathbb{H}\subset \mathbb{K}$.
Namely, it holds:
\begin{thm}[Neumark] \label{thm61}
If $\omega \mapsto E(\omega)$ is an $\mathcal{L} (\mathbb{H})$~-~valued measure on the $\sigma$~-~algebra $\Sigma$, such that $ 0 \leq E(\omega) \leq I_\mathbb{H}$, then there exist a locally Hilbert space $\mathbb{K}$, which includes $\mathbb{H}$ as a locally Hilbert subspace and an $\mathcal{L} (\mathbb{K})$~-~valued measure $\omega \mapsto F(\omega)$, such that $F(\omega)$ are self-adjoint projections on $\mathbb{K}$ and
\begin{equation} \label{eq61}
F(\omega) = J^* E(\omega) J ,\quad \omega \in \Sigma ,
\end{equation}
$J$ being the (locally isometric) embedding of $\,\mathbb{H}$ into $\mathbb{K}$.
\end{thm}
\begin{proof}
By putting $S = \Sigma$, the intersection as addition and the involution $\omega^* = \omega,$  $\,\omega \in \Sigma$ and, applying Theorem \ref{thm51}, the statement is easily inferred.
\end{proof}

\paragraph2 For a \emph{locally contraction} $T$ on $\mathbb{H}$, i.e. $I - T^* T$ is positive in $\mathcal{L} (\mathbb{H})$, by putting
\begin{equation} \label{eq62}
T^{(n)}: = \begin{cases}  T^n , & n \in \mathbb{Z}_+ \\
T^{* - n} , & \text{ elsewhere }
\end{cases}
,
\end{equation}
$S =\mathbb{Z}$ and $n^* :\ = -n ,\ n \in \mathbb{Z}$, and applying Theorem \ref{thm51} we obtain a locally unitary minimal dilation $U = \pi (1)$ on a minimal larger locally Hilbert space $\mathbb{K}$, i.e. $T^n=J^*U^nJ$, $n\in\mathbb{Z}_{+}$, $\mathbb{K}=\bigvee\limits_{n \in \mathbb{Z}}U^nJ\mathbb{H}$.

\paragraph3 If $\{T_t\}_{t\in\mathbb{R}_{+}}$ is a \emph{locally contraction semigroup} on $\mathbb{H}$, then by defining

\begin{equation*}
T_{(t)}: = \begin{cases}  T_t , & t \in \mathbb{R}_+ \\
T^{*}_{ (- t)} , & t<0
\end{cases}
,
\end{equation*}
the function $$\mathbb{R}\ni t\mapsto T_{(t)}\in\mathcal{L}(\mathbb{H})$$
will be LPD on the group $\mathbb{R}$ and  satisfies the locally boundedness condition. By applying Theorem \ref{thm51}, we get the existence of a locally unitary group $U_t:=\pi(t)$, $t\in\mathbb{R}$ on $\mathbb{H}$, which dilates $\{T_t\}_{t\in\mathbb{R}_{+}}$. It is also possible to obtain dilation results for a semigroup $\{T_s\}_{s\in S}$ of locally contractions from $\mathcal{L}(\mathbb{H})$, where $S$ is an abelian subsemigroup of a group $G$, with $S\cap S^{-1}=\{e\}$ and $s^*=s^{-1}$, $s\in S$  (see \cite{Su.1975} for the Hilbert space frame).

\paragraph4 If $T = \varinjlim\limits_{\lambda\in\Lambda} T_\lambda \in \mathcal{L} (\mathbb{H})$ satisfies the locally condition
\begin{equation} \label{eq63}
\parallel{p(T_\lambda)}\parallel_\lambda \leq \sup\limits_{|{z}| \le 1} |{\rho p (z) + (1 - \rho) p (0)}|
\end{equation}
for each polynomial $p$, then $T$ has a unitary $\rho$~-~dilation $U$ on a (minimal) larger locally Hilbert space $\mathbb{K}$, i.e. $T^n=\rho J^*U^nJ$, $n=1,2,\dots$ and $\mathbb{K}=\bigvee\limits_{n \in \mathbb{Z}}U^nJ\mathbb{H}$. Indeed, it results, applying Theorem \ref{thm52}, that the above condition is in fact equivalent to the conditions ($\rho$LPD) and ($\rho$LBC) for $S = \mathbb{Z}$, as above, with
\begin{equation*}
T(n) :\ = \begin{cases} \frac{1}{\rho} T^{(n)} ,& n \ne 0 \\
I , & n = 0 \end{cases} , n \in \mathbb{Z} ,
\end{equation*} (compare also with \cite{Ga.1970}).
It is clear that such a $\rho$-contraction $T$ is a bounded element in the $LC^*$-algebra ${\mathscr L}({\mathbb H})$, i.e. $T \in \Bigl({\mathscr L}({\mathbb H})\Bigr)_b$.
 It is not hard to see that for $T = \varinjlim\limits_{\lambda\in\Lambda} T_\lambda$ the following conditions are equivalent:
\begin{itemize}
\item[(i)] $T$ has a locally unitary $\rho$~-~dilation;
\item[(ii)] $T$ satisfies the condition (\ref{eq63});
\item[(iii)] each $T_\lambda$ has a unitary $\rho$~-~dilation, $\lambda\in\Lambda$.
\end{itemize}

\paragraph5 It is also possible to obtain from Theorem \ref{thm51} an analogue of the Bram criteria on an operator from $\mathcal{L}(\mathbb{H})$ for the existence of a normal extension. It means that the notion of
a subnormal operator in $\mathcal{L}(\mathbb{H})$ makes sense as in the Hilbert space case. Moreover, the particular case of the existence of a locally unitary operator can be obtained by applying {\bf 2} to a locally isometry. \\
Similar results, as in the Hilbert space case, can be obtained for commuting systems of operators on a locally Hilbert space.



\subsection*{Acknowledgment}
This work was entirely supported by research grant 2-CEx06-11-34/25.07.06.


\begin{thebibliography}{1}
\bibitem{Am.2002}
M. Amini, \textit{Non commutative topology and local structure of operator algebras}, arXiv:math/0205257v1 [math.OA].

\bibitem{Ap.1971}
C. Apostol, \textit{$B^{*}$~-~algebras and their representations}, J. London Math. Soc. \textbf{33} (1971), 30--38.

\bibitem{Fr.1986}
M. Fragoulopoulou, \textit{On locally $W^{*}$~-~algebras}, Yokohama Math. J. \textbf{34} (1986), 35--51.

\bibitem{Fr.1988}
M. Fragoulopoulou, \textit{An introduction to the representation theory of topological $*$~-~algebras}, Schriftenreihe Math. Inst. Univ. M\"{u}nster \textbf{48} (1988), 1--81.

\bibitem{Ga.Ga.2007}
D. Ga\c spar, P. Ga\c spar, \textit{Reproducing kernel Hilbert modules over locally $C^{*}$~-~algebras}, An. Univ. Timi\c soara Ser. Mat.-Inform. \textbf{XLV} (2007), 245--252.

\bibitem{Ga.1970}
D. Ga\c spar, \textit{Spectral $\rho$~-~dilations for the representation of function algebras}, An. Univ. Timi\c soara Ser. Mat.-Inform. \textbf{VIII} (1970), 153--157.

\bibitem{In.1971}
A. Inoue, \textit{Locally $C^{*}$~-~algebras}, Mem. Faculty Sci. Kyushu Univ. Ser. A \textbf{25} (1971), 197--235.

\bibitem{Jo.2006}
M. Joi\c ta, \textit{Hilbert modules over locally $C^{*}$~-~algebras}, Ed. Univ. Bucure\c sti, 2006.

\bibitem{Ma.1986}
A. Mallios, \textit{Topological algebras}, Selected Topics, North Holland, 1986.

\bibitem{Ph.1988}
N.C. Philips, \textit{Inverse limits of $C^{*}$~-~algebras}, J. Operator Theory \textbf{19} (1988), 159--195.

\bibitem{Sc.1975}
K. Schm\"{u}dgen, \textit{\"{U}ber $LMC^{*}$~-~Algebren}, Math. Nachr \textbf{68} (1975), 167--182.

\bibitem{Se.1979}
Z. Sebestyen, \textit{Every $C^{*}$~-~seminorm is automatically submultiplicative}, Period. Math. Hungar \textbf{10} (1979), 1--8.

\bibitem{Su.1975}
I. Suciu, \textit{Function Algebras}, Ed. Acad., Bucure\c sti, 1975.

\bibitem{Sz.1955}
B. Sz.-Nagy, \textit{Prolongements des transformations de l'espace de Hilbert qui sortent de cet espace}, Appendix at F. Riesz and B. Sz.-Nagy "Le\c{c}ons d'analyse fonctionelle", Akademiai Kiado, Budapest, 1955.

\bibitem{Zh.Sh.2001}
Yu. I. Zhuraez, F. Sharipov, \textit{Hilbert modules over locally $C^{*}$~-~algebras}, 	arXiv:math/0011053v3 [math.OA].

\end{thebibliography}
\end{document}